\documentclass[preprint,3p,sort]{elsarticle}

\usepackage{amsfonts}
\usepackage{amsthm}
\usepackage{graphicx}
\usepackage{epstopdf}
\usepackage{algorithmic}
\usepackage{amsmath}
\usepackage{amssymb}
\usepackage{hyperref}
\usepackage[T1]{fontenc}
\usepackage{soul,color}
\usepackage{cases}
\newtheorem{theorem}{Theorem}[section]
\newtheorem{proposition}[theorem]{Proposition}
\newtheorem{lemma}[theorem]{Lemma}

\newtheorem{definition}{Definition}[section]
\theoremstyle{remark}
\newtheorem{remark}[theorem]{Remark}
\theoremstyle{example}
\newtheorem{example}[theorem]{Example}

\numberwithin{equation}{section}
\begin{document}


\begin{frontmatter}

%

\title{Nonlinear Boundary Conditions for Energy and Entropy Stable Initial Boundary Value Problems in Computational Fluid Dynamics}

\author[sweden,southafrica]{Jan Nordstr\"{o}m}
\cortext[secondcorrespondingauthor]{Corresponding author}
\ead{jan.nordstrom@liu.se}
\address[sweden]{Department of Mathematics, Applied Mathematics, Link\"{o}ping University, SE-581 83 Link\"{o}ping, Sweden}
\address[southafrica]{Department of Mathematics and Applied Mathematics, University of Johannesburg, P.O. Box 524, Auckland Park 2006, Johannesburg, South Africa}

\begin{abstract}
We derive new boundary conditions and implementation procedures for nonlinear initial boundary value problems that lead to energy and entropy bounded solutions. A  step-by-step procedure for general nonlinear hyperbolic problems on skew-symmetric form is presented. That procedure is subsequently applied to the three most important equations in computational fluid dynamics: the shallow water equations and the incompressible and compressible Euler equations. Both strong and weak imposition of the nonlinear boundary conditions are discussed. Based on the continuous analysis, we show that the new nonlinear boundary procedure lead to energy and entropy stable discrete approximations if the scheme is formulated on summation-by-parts form in combination with a weak implementation of the boundary conditions.
\end{abstract}



\end{frontmatter}


\section{Introduction}

In this paper we will complete the general stability theory for nonlinear hyperbolic initial boundary value problems (IBVPs) partly developed in \cite{nordstrom2022linear-nonlinear,Nordstrom2022_Skew_Euler}. This theory is valid for both linear and nonlinear primal and dual problems. It is direct, easy to understand and leads to $L_2 $ estimates. The requirement for an energy and entropy bound is that  {\it i)} a skew-symmetric form of the governing equations exist and {\it ii)} energy bounding boundary conditions (BCs) are available. In  \cite{nordstrom2022linear-nonlinear, Nordstrom2022_Skew_Euler} we focused on the skew-symmetric property {\it assuming}  that boundary conditions leading to an energy bound were available. In this article we derive these BCs explicitly, and show how to implement them in a provable stable way. We exemplify the procedure for the most important equations in computational fluid dynamics (CFD): the shallow water equations (SWEs), the incompressible Euler equations (IEEs) and the compressible Euler equations (CEEs).

It was shown in \cite{nordstrom2022linear-nonlinear} that the original form of the velocity-divergence form of the IEEs equations had the required skew-symmetric form and that it could be derived for the SWEs. In  \cite{Nordstrom2022_Skew_Euler} we showed that also the CEEs could be transformed to skew-symmetric form. It was also shown that the new skew-symmetric formulation allows for a mathematical (or generalised) entropy conservation and bound. Once the skew-symmetric formulation is obtained, an energy and entropy bound follows by applying integration-by-parts (IBP) and imposing proper boundary conditions.  The continuous procedure was reused by discretising the equations in space using summation-by-parts (SBP) operators \cite{svard2014review,fernandez2014review} which discretely mimic the IBP procedure.  To derive the stable boundary procedures that was {\it assumed} to exist in \cite{nordstrom2022linear-nonlinear, Nordstrom2022_Skew_Euler} is the topic of this paper. As in the previous papers, it is shown that the key to stability is found in the continuous formulation.

Skew-symmetric formulations for parts or the whole set of governing flow equations have drawn interest previously \cite{Petr2007,Rozema2014,Sesterhenn2014,Halpern2018} where fragments of the general theory in \cite{nordstrom2022linear-nonlinear,Nordstrom2022_Skew_Euler} was included. Nonlinear boundary conditions were not discussed. With few exceptions, only boundary conditions for solid walls (or glancing boundaries) have been considered previously as for example in \cite{svard2014,parsani2015entropy,svar2018,svard2021,chan2022,Gjesteland2022}. Solid wall boundary conditions are notoriously simple and straightforward to implement due to their homogeneous nature, i.e. no external non-zero data must be considered. In contrast to the previous investigations, we will for the first time (to the best of our knowledge) treat the general case with non-homogeneous nonlinear boundary conditions and derive estimates of the solution in terms of given non-zero boundary data.


The remaining part of paper is organised as follows: In Section~\ref{sec:theory} we reiterate and complement the main theoretical findings in  \cite{nordstrom2022linear-nonlinear,Nordstrom2022_Skew_Euler} and outline the general procedure for obtaining energy and entropy bounds. The remaining key ingridient: how to formulate and impose general nonlinear boundary conditions, is presented in Section~\ref{BC_theory}. In Section~\ref{Examples}, we show that the most important IBVPs in CFD: the SWEs, the IEEs and the CEEs can be described by the new general theoretical framework. Explicit examples of boundary conditions and implementation procedures are given for all three cases. In Section~\ref{numerics} we return to the general formulation and show that the energy and entropy bounded continuous formulation lead to  nonlinear stability of the SBP based semi-discrete scheme, including non-zero boundary data. A summary is provided in Section~\ref{sec:conclusion}.

\section{Nonlinear energy and entropy boundedness: the governing equations}\label{sec:theory}

Following \cite{nordstrom2022linear-nonlinear,Nordstrom2022_Skew_Euler}, we consider the following general hyperbolic IBVP
\begin{equation}\label{eq:nonlin}
P U_t + (A_i(V) U)_{x_i}+B_i(V)U_{x_i}+C(V)U=0,  \quad t \geq 0,  \quad  \vec x=(x_1,x_2,..,x_k) \in \Omega
\end{equation}
augmented with the initial condition $U(\vec x,0)=F(\vec x)$ in $\Omega$ and  the non-homogeneous boundary condition
\begin{equation}\label{eq:nonlin_BC}
L(V) U = g(\vec x,t),  \quad t \geq 0,  \quad  \vec x=(x_1,x_2,..,x_k) \in  \partial\Omega.
\end{equation}
In (\ref{eq:nonlin_BC}), $L$ is the boundary operator and $g$ the boundary data. In (\ref{eq:nonlin}), Einsteins summation convention is used and $P$ is a symmetric positive definite (or semi-definite) time-independent matrix that defines an energy norm (or semi-norm) $\|U\|^2_P= \int_{\Omega} U^T P U d\Omega$. We assume that $U$ and $V$ are smooth. The $n \times n$ matrices $A_i,B_i,C$ are smooth functions of the $n$ component vector $V$, but otherwise arbitrary. Note that (\ref{eq:nonlin}) and (\ref{eq:nonlin_BC}) encapsulates both linear ($V \neq U$) and nonlinear  ($V=U$) problems. 


\begin{definition}
Firstly, the problem (\ref{eq:nonlin}) is energy conserving if $\|U\|^2_P= \int_{\Omega} U^T P U d\Omega$ only changes due to boundary effects. Secondly, it is energy bounded if $\|U\|^2_P  \leq \|F\|^2_P$ for a minimal number of homogeneous $(g=0)$ boundary conditions (\ref{eq:nonlin_BC}). 
Thirdly, it is strongly energy bounded if $\|U\|^2_P  \leq \|F\|^2_P+  \int_0^t (\oint G^TG \ ds) dt$ for a minimal number of non-homogeneous $(g \neq 0)$ boundary conditions (\ref{eq:nonlin_BC}), where $G=G(g,\vec x,t)$.
\end{definition}

\begin{proposition}\label{lemma:Matrixrelation}
The IBVP  (\ref{eq:nonlin}) for linear ($V \neq U$) and nonlinear  ($V=U$)  is energy conserving if
\begin{equation}\label{eq:boundcond}
B_i= A_i^T, \quad i=1,2,..,k \quad \text{and } \quad C+C^T = 0
\end{equation}
holds. It is energy bounded if it is energy conserving and the boundary conditions (\ref{eq:nonlin_BC}) for $g=0$ lead to
\begin{equation}\label{1Dprimalstab}
\oint\limits_{\partial\Omega}U^T  (n_i A_i)   \\\ U \\\ ds = \oint\limits_{\partial\Omega} \frac{1}{2} U^T ((n_i A_i)  +(n_i A_i )^T) U \\\ ds \geq 0.
\end{equation}
It is strongly energy bounded if it is energy conserving and the boundary conditions (\ref{eq:nonlin_BC}) for $g \neq 0$ lead to
\begin{equation}\label{1Dprimalstab_strong}
\oint\limits_{\partial\Omega}U^T  (n_i A_i)   \\\ U \\\ ds = \oint\limits_{\partial\Omega} \frac{1}{2} U^T ((n_i A_i)  +(n_i A_i )^T) U \\\ ds \geq - \oint\limits_{\partial\Omega} G^TG \\\ ds,
\end{equation}
where $G=G(g,\vec x,t)$ is independent of the solution $U$.
\end{proposition}
\begin{proof}
The energy method applied to (\ref{eq:nonlin}) yields
\begin{equation}\label{eq:boundaryPart1}
\frac{1}{2} \frac{d}{dt}\|U\|^2_P + \oint\limits_{\partial\Omega}U^T  (n_i A_i)  \\\ U \\\ ds= \int\limits_{\Omega}(U_{x_i}^T  A_i U - U^T B_i U_{x_i}) \\\ d \Omega -\int\limits_{\Omega} U^T  C U \\\ d \Omega,
\end{equation}
where $(n_1,..,n_k)^T$ is the outward pointing unit normal. The terms on the right-hand side of (\ref{eq:boundaryPart1}) are cancelled by (\ref{eq:boundcond}) leading to energy conservation. If in addition (\ref{1Dprimalstab}) or (\ref{1Dprimalstab_strong}) holds, an energy bound or a strong energy bound respectively follows after integration in time. 
\end{proof}
\begin{remark}\label{necessary}
For linear problems, a minimal number of boundary conditions that lead to a bound is a necessary and sufficient condition for well-posedness. For nonlinear problems this is not the case. A minimal number of boundary conditions that lead to a bound is a necessary but not a sufficient condition \cite{nordstrom2005,nordstrom2020,nordstrom2021linear}.
\end{remark}

For non-smooth solutions $U$, (\ref{eq:nonlin}) interpreted in a weak sense allows for an entropy conservation law.
\begin{proposition}\label{:Entropy-relation}
The IBVP  (\ref{eq:nonlin}) together with conditions (\ref{eq:boundcond}) leads to the entropy conservation law
\begin{equation}\label{eq:entropy-equation}
S_t + (F_i)_{x_i} = 0,
\end{equation}
where $S=U^T P U/2$ is the mathematical (or generalised) entropy  and  $F_i=U^T A_i U$ are the entropy fluxes.
\end{proposition}
\begin{proof}
Multiplication of  (\ref{eq:nonlin}) from the left with $U^T$ yields
\begin{equation}\label{eq:entropy-der}
 (U^T P U/2)_t  +  (U^T A_i U)_{x_i}  = (U_{x_i}^T  A_i U - U^T B_i U_{x_i}) - U^T  C U.
\end{equation}
The right-hand side of (\ref{eq:entropy-der}) is cancelled by (\ref{eq:boundcond}) leading to the entropy conservation relation (\ref{eq:entropy-equation}).
\end{proof}
\begin{remark}\label{entropy-comments}
The entropy conservation law (\ref{eq:entropy-equation}) holds for smooth solutions. For discontinuous solutions it holds in a distributional sense. The non-standard compatibility conditions in this case reads
\begin{equation}\label{eq:entropy-der}
\partial S/\partial U=S_{U}=U^T P,  \quad S_{U} P^{-1}( (A_i(V) U)_{x_i}+A^T_i(V)U_{x_i}+C(V)U)=(U^T A_i U)_{x_i}.
\end{equation}
The entropy $S$ is convex ($S_{UU}=P$) and identical to the energy \cite{nordstrom2021linear}. 
In the following we will use energy to denote both quantities, but sometimes remind the reader by writing out both notations explicitly.
\end{remark}


\section{Nonlinear energy and entropy boundedness: the boundary conditions}\label{BC_theory} 
We start with a couple of convinient transformations. Consider the boundary term
\begin{equation}\label{1Dprimalstab_trans}
\oint\limits_{\partial\Omega}U^T  (n_i A_i)   \\\ U \\\ ds = \oint\limits_{\partial\Omega} \frac{1}{2} U^T ((n_i A_i)  +(n_i A_i )^T) U \\\ ds = \oint\limits_{\partial\Omega}U^T   \tilde A(V)   \\\ U \\\ ds,
\end{equation}
where $\tilde A(V)$ is symmetric. Recall that if $V=U$ we are dealing with a nonlinear problem, otherwise a variable coefficient problem. In the CFD problems we consider, the Cartesian velocity
field is transformed to the normal and tangential ones leading to the new vectors $U_n=NU$. Next we rotate the matrix $\tilde A$ to diagonal form as $ \tilde T^T \tilde A  \tilde T =  \Lambda = diag( \lambda_i)$
which gives us new rotated variables $W = (N  \tilde T)^{-1}U=T^{-1} U$ and
\begin{equation}\label{1Dprimalstab_trans_final}
\oint\limits_{\partial\Omega}U^T   \tilde A(V)   \\\ U \\\ ds  \\\ ds = \oint\limits_{\partial\Omega}W^T   \Lambda   \\\ W\\\ ds = \oint\limits_{\partial\Omega}(W^+)^T   \Lambda^+   \\\ W^+ + (W^-)^T   \Lambda^-   \\\ W^-\\\ ds=\oint\limits_{\partial\Omega}  \lambda_i W_i^2 \\\ ds,
\end{equation}
where we again use Einsteins summation convention. In (\ref{1Dprimalstab_trans_final}), $\Lambda^+$ and  $\Lambda^-$ denote the positive and negative parts of $\Lambda$ respectively, while $W^+$  and $W^-$ denote the corresponding variables.  The new rotated variables $W=W(U)$ are functions of the solution in both the linear and nonlinear case. In the nonlinear case,  the diagonal matrix  $\Lambda(U)$ is solution dependent and not a priori bounded while in the linear case, $\Lambda(V)$ is bounded by external data. This difference lead to significant differences in the boundary condition procedure.
\begin{remark}\label{Sylvester}
For linear problems, the number of boundary conditions is equal to the number of eigenvalues of $\tilde A(V)$ with the wrong (in this case negative) sign  \cite {nordstrom2020}. Sylvester's Criterion \cite{horn2012}, show that the number of boundary conditions is equal to the number of $\lambda_i(V)$  with the wrong sign if the rotation matrix  $T$ is non-singular.  In the nonlinear case where $\lambda_i=\lambda_i(U)$ it is more complicated since multiple forms of the boundary term $W^T   \Lambda W$ may exist,  see Section \ref{Eulerex} below and \cite{nordstrom2022linear-nonlinear,Nordstrom2022_Skew_Euler,nordstrom2021linear} for examples. With a slight abuse of notation we will sometimes refer to $\Lambda(U)$ as "eigenvalues" and to the rotated variables $W(U)$ as "characteristic" variables, although strictly speaking they are not, even though they play a similar role.
\end{remark}

We will impose the boundary conditions both strongly and weakly. For the weak imposition we introduce a lifting operator $L_C$ that enforce the boundary conditions in our governing equation (\ref{eq:nonlin}) as follows
\begin{equation}\label{eq:nonlin_lif}
P U_t + (A_i(V) U)_{x_i}+A^T_i(V)U_{x_i}+C(V)U+L_C(L(V)U-g)=0,  \quad t \geq 0,  \quad  \vec x=(x_1,x_2,..,x_k) \in \Omega.
\end{equation}
The lifting operator for two smooth vector functions  $\Phi, \Psi$ satisfies $\int\limits \Phi^T   L_C(\Psi) d \Omega = \oint\limits \Phi^T  \Psi ds$
which enables development of the essential parts of the numerical boundary procedure in the continuous setting \cite{Arnold20011749,nordstrom_roadmap}.

\subsection{The general form of nonlinear boundary conditions in rotated variables}\label{nonlinear_BC_char}
The starting point for the derivation of stable general nonlinear (and linear)  boundary conditions (\ref{eq:nonlin_BC}) is the form (\ref{1Dprimalstab_trans_final}) of the boundary term. First we need to find the formulation (\ref{1Dprimalstab_trans_final}) with a {\it minimal} number of entries in $\Lambda^-$ \cite{nordstrom2022linear-nonlinear,Nordstrom2022_Skew_Euler,nordstrom2021linear} (there might be more than one formulation of the cubic boundary terms). Next, we need to specify the characteristic variables $W^-$ in terms of $W^+$ and external data. The general form is
\begin{equation}\label{Gen_BC_form}
 S(W^--R W^+) =G  \quad  \text{or equivalently} \quad  W^- = R W^+ + S^{-1}G.
\end{equation}
In (\ref{Gen_BC_form}), $S$ is a non-singular matrix combining values of $W^-$, the matrix $SR$ combine values of $W^+$ while $G$ is given external data. The boundary condition (\ref{Gen_BC_form}) implemented weakly using a lifting operator is
\begin{equation}\label{Pen_term_Gen_BC_form}
L_C= L_C(2( J^-T^{-1})^T \Sigma (W^- - R W^+ - S^{-1}G)),
\end{equation}
where $W=T^{-1}U$, $W^-=J^-W$, $W^+=J^+W$ and $\Sigma$ is a penalty matrix. 
 After the derivation of the stability conditions we will return to the boundary condition formulation (\ref{eq:nonlin_BC}) in the original variables.

\subsection{Boundary conditions and implementation techniques for stability of nonlinear problems}\label{summary_BC_theory}
Before attacking the nonlinear problem, we digress momentarily to the linear case to introduce one aspect of our subsequent nonlinear analysis. 
In the simplest possible version of (\ref{Gen_BC_form}) one can specify $W^-=g$ corresponding to negative $\lambda_i(V)$ indicated by $\lambda_i^-=-|\lambda_i(V)|$. Since $|\lambda_i(V)|$ are bounded, we obtain
\begin{equation}\label{1Dprimalstab_trans_final_extra}
\oint\limits_{\partial\Omega}U^T   \tilde A(V)   \\\ U \\\ ds  \\\  = \oint\limits_{\partial\Omega}W^T   \Lambda   \\\ W\\\ ds =  \oint\limits_{\partial\Omega}  (W^+)^T   \Lambda^+   \\\ W^+ + g^T   \Lambda^-   \\\ g\\\ ds \geq  - \oint\limits_{\partial\Omega}  G^TG \\\ ds,
\end{equation}
where $G_i = \sqrt{|\lambda_i^-(V)|}g_i$. Hence we get a strong energy bound  in terms of external data. However, in the nonlinear case, no estimate is obtained since $\lambda_i^-(U)$ is not a priori bounded, see \cite{nordstrom2019, nordstrom2020spatial} for IEE examples.

The procedure to arrive at  a general stable nonlinear inhomogeneous boundary condition and implementation consist of the following steps for the unknowns in $R,S,\Sigma$ in (\ref{Gen_BC_form}).
\begin{enumerate}

\item Derive strong homogeneous ($G=0$) boundary conditions. This lead to conditions on matrix $R$.

\item Derive strong inhomogeneous ($G\neq 0$) boundary conditions.  This lead to conditions on matrix $S$.

\item Derive weak homogeneous ($G=0$) boundary conditions. This lead to conditions on matrix $\Sigma$.

\item Show that the weak inhomogeneous ($G\neq 0$) case of the  boundary conditions follow from 1-3 above.

\end{enumerate}
The following Lemma (structured as the step-by-step procedure above) is the main result of this paper. 
\begin{lemma}\label{lemma:GenBC}
Consider the boundary term described in (\ref{1Dprimalstab_trans}),(\ref{1Dprimalstab_trans_final}) and the boundary conditions (\ref{Gen_BC_form}) implemented strongly or weakly using (\ref{Pen_term_Gen_BC_form}).
Furthermore, let $ |\Lambda^- |=diag( |\lambda^-_i |)$ and $  |\Lambda^- |^{1/2}=diag(  \sqrt{|\lambda_i^-|})$. 

The boundary term augmented with {\bf 1. strong nonlinear homogeneous boundary conditions} is positive semi-definite if the matrix $R$ is such that
\begin{equation}\label{R_condition}
\Lambda^+ - R^T  |\Lambda^- |  R  \geq 0.
\end{equation}

The boundary term augmented with {\bf 2. strong nonlinear inhomogeneous boundary conditions} is bounded by external given data if
the matrix  $R$ satisfies (\ref{R_condition}) with strict inequality and the matrix $S$ satisfies 
\begin{equation}\label{S_condition}
S =   \tilde S^{-1} |\Lambda^- |^{1/2} \ \mbox{with 
$\tilde S$ sufficiently small.}
\end{equation}

The boundary term augmented with {\bf 3. weak nonlinear homogeneous boundary conditions} is positive semi-definite if the matrix $R$
satifies (\ref{R_condition}) and the matrix $\Sigma$ satisfies
\begin{equation}\label{Sigma_condition}
\Sigma=|\Lambda^- |.
\end{equation}

The boundary term augmented with {\bf 4. weak nonlinear inhomogeneous boundary conditions} is bounded by external given data if
the matrix $R$ satisfies (\ref{R_condition}) with strict inequality,  the matrix $S$ satisfies  (\ref{S_condition}) and the matrix $\Sigma$ satisfies (\ref{Sigma_condition}).
\end{lemma}
\begin{proof} We proceed in the step-by-step manner described above.

{\it 1. The homogeneous boundary condition (\ref{Gen_BC_form}) implemented strongly} (with $G=0$) lead to $W^T   \Lambda   W  = (W^+)^T ( \Lambda^+ - R^T  |\Lambda^- |  R)W^+$ and (\ref{R_condition})
lead to a positive semi-definite boundary term.

{\it 2. The inhomogeneous boundary condition (\ref{Gen_BC_form}) implemented strongly} (with $G\neq0$)  lead to
\begin{equation}\label{S1_derivation}
 W^T   \Lambda   W  = (W^+)^T  \Lambda^+ W^+ -  (W^+ + S^{-1} G)^T  |\Lambda^- |  (W^+ + S^{-1} G).
\end{equation}
Expanding (\ref{S1_derivation}), adding and subtracting $G^T G$ and using $S$ as in (\ref{S_condition}) lead to the result
\begin{equation}
\label{estimate_2}
 W^T   \Lambda   W =
\begin{bmatrix}
W^+ \\
G
\end{bmatrix}^T
\begin{bmatrix}
\Lambda^+ - R^T  |\Lambda^- |  R & - R^T  |\Lambda^- |^{1/2} \tilde S \\
-\tilde S^T  |\Lambda^- |^{1/2} R  &  I-\tilde S^T \tilde S
\end{bmatrix}
\begin{bmatrix}
W^+ \\
G
\end{bmatrix} - G^T G,
\end{equation}
which is bounded from below by external data if  $\tilde S$ is sufficiently small and (\ref{R_condition}) holds strictly.

{\it 3. The homogeneous boundary condition (\ref{Gen_BC_form}) implemented weakly} (with $G=0$)  using the lifting operator in (\ref{Pen_term_Gen_BC_form}) lead to the boundary term
\begin{equation}\label{Sigma1_derivation}
 W^T \Lambda   W  + 2 U^T  ( J^-T^{-1})^T \Sigma (W^- - R W^+))= W^T   \Lambda   W^ + + 2 (W^-)^T  \Sigma (W^- - R W^+).
 \end{equation}
 Collecting similar terms transforms the right hand side to
\begin{equation}\label{Sigma2_derivation}
 (W^+)^T   \Lambda^+  W^+ +( W^-)^T(- |\Lambda^- | + 2 \Sigma) W^- -2 ( W^-)^T \Sigma R W^+.
\end{equation}
The choice (\ref{Sigma_condition}) of $\Sigma$ followed by adding and subtracting  $(R W^+)^T  |\Lambda^- | R W^+$ transform  (\ref{Sigma2_derivation}) into
\begin{equation}\label{Sigma4_derivation}
  (W^+)^T ( \Lambda^+ - R^T  |\Lambda^- |  R)W^+ + (W^--RW^+)^T |\Lambda^- | (W^--R W^+),
\end{equation}
which lead to positive semi-definite boundary term by using condition (\ref{R_condition}).

{\it The inhomogeneous boundary condition (\ref{Gen_BC_form}) implemented weakly} (with $G\neq0$)  using the lifting operator in (\ref{Pen_term_Gen_BC_form}) and the choice $\Sigma$ in (\ref{Sigma_condition}) lead to the boundary terms
\[
W^T \Lambda   W +  2 U^T  ( J^-T^{-1})^T \Sigma (W^- - R W^+ - S^{-1} G)) =W^T \Lambda   W   + 2 (W^-)^T  |\Lambda^- | (W^- - R W^+  -  S^{-1}G)).
\]
By adding and subtracting $(W^-)^T  |\Lambda^- | (W^-)$  and rearranging, the boundary terms above can be written as
\begin{equation}\label{final _gen_weak_result_1}
(W^+)^T (\Lambda^+ - R^T  |\Lambda^- | R) W^+ +   (W^- - R W^+)^T  |\Lambda^- | (W^- - R W^+) - 2 (W^-)^T  |\Lambda^- |  S^{-1}G.
\end{equation}
By rearranging  (\ref{final _gen_weak_result_1}) we find that it is equivalent to
\begin{equation}\nonumber
 (W^- - R W^+ - S^{-1} G)^T  |\Lambda^- | (W^- - R W^+ - S^{-1} G) + 
 \begin{bmatrix}
W^+ \\
G
\end{bmatrix}^T
\begin{bmatrix}
\Lambda^+ - R^T  |\Lambda^- |  R & - R^T  |\Lambda^- |  S^{-1}  \\
-( S^{-1} )^T  |\Lambda^- | R  &  - (S^{-1} )^T  |\Lambda^- | S^{-1} 
\end{bmatrix}
\begin{bmatrix}
W^+ \\
G
\end{bmatrix}.
\end{equation}
The first left term is obviously positive semi-definite. By adding and subtracting the boundary data $G^T G$ and inserting the matrix  $S$ as in (\ref{S_condition}), the second right term becomes
\begin{equation}\label{final _gen_weak_result_3}
 \begin{bmatrix}
W^+ \\
G
\end{bmatrix}^T
\begin{bmatrix}
\Lambda^+ - R^T  |\Lambda^- |  R & - R^T  |\Lambda^- |^{1/2} \tilde S \\
-\tilde S^T  |\Lambda^- |^{1/2} R  &  I-\tilde S^T \tilde S
\end{bmatrix}
\begin{bmatrix}
W^+ \\
G
\end{bmatrix} - G^T G, 
\end{equation}
which is bounded from below by external data if $\tilde S$ is sufficiently small and condition (\ref{R_condition}) holds strictly.
\end{proof}
Lemma \ref{lemma:GenBC} can be used to prove that the estimates (\ref{1Dprimalstab}) and (\ref{1Dprimalstab_strong}) in Proposition \ref{lemma:Matrixrelation} holds.

\subsection{The general form of nonlinear boundary conditions in original variables}\label{nonlinear_BC_orig}
We are now ready to connect the characteristic boundary condition formulation (\ref{Gen_BC_form}) with (\ref{eq:nonlin_BC}) in the original variables.  By using  the definitions $W=T^{-1}U$, $W^-=J^-W$, $W^+=J^+W$ and relation (\ref{S_condition}) we find that (\ref{Gen_BC_form}) transforms to
\begin{equation}\label{Gen_BC_form_expanded}
\tilde S^{-1}  |\Lambda^- |^{1/2} (J^- - R J^+)T^{-1}U =G.
\end{equation}
By comparing (\ref{eq:nonlin_BC}) and (\ref{Gen_BC_form_expanded}), the original boundary operator and boundary data can be identified as
\begin{equation}\label{eq:nonlin_BC_details}
L=|\Lambda^- |^{1/2} (J^- - R J^+)T^{-1}   \ \mbox{and}   \  g= \tilde S G
\end{equation}
respectively. This concludes the analysis of the general formulation of nonlinear boundary conditions.


\section{Application of the general theory to initial boundary value problems in CFD}\label{Examples}
We will specifically consider the IEEs, the  SWEs and the CEEs, and focus on the boundary conditions. 

\subsection{The 2D incompressible Euler equations}\label{Eulerex}
The incompressible 2D Euler equations in split form are
\begin{equation}
P U_t 
+ \frac{1}{2}\left[ (A U)_x + A U_x 
+  (B U)_y + B U_y \right]
= 0.
\label{NS_splitting}
\end{equation} 
where $U=(u,v,p)^T$ and
\begin{equation}
P =
\begin{bmatrix}
1 & 0 & 0 \\
0 & 1 & 0 \\
0 & 0 & 0
\end{bmatrix},
\quad
A=
\begin{bmatrix}
u & 0 & 1 \\
0 & u & 0 \\
1 & 0 & 0
\end{bmatrix},
\quad
B=
\begin{bmatrix}
v & 0 & 0 \\
0 & v & 1 \\
0 & 1 & 0
\end{bmatrix}.
\label{ABI}
\end{equation}
Since the matrices $A,B$ are symmetric, the formulation (\ref{NS_splitting}) is in the required skew-symmetric form (\ref{eq:nonlin_lif})
We obtain an estimate in the semi-norm  $\|U\|^2_P= \int_{\Omega} U^T P  U d \Omega$ involving only the velocities. Note that the pressure $p$ includes a division by the constant density, and hence has the dimension velocity squared.

By applying the transformation $W = T^{-1}U$ described above, the boundary term gets the form
\begin{equation}\label{1Dprimalstab_trans_final_icomp}
U^T (n_1 A + n_2 B) = W^T  \Lambda W= (W^+)^T   \Lambda^+   W^+ + (W^-)^T   \Lambda^- W^-
\end{equation}
where $W=(u_n+p/u_n, u_{\tau}, p/u_n)^T$,  $\Lambda=diag(u_n, u_n, -u_n)$, $u_n=n_1u+n_2v$ and  $ u_{\tau}=-n_2u+n_1v$.  At inflow 
\begin{equation}
W^-=
\begin{bmatrix}
u_n+p/u_n\\
u_{\tau} 
\end{bmatrix},
\quad
\Lambda^- =
\begin{bmatrix}
u_n & 0 \\
0 & u_n 
\end{bmatrix},
\quad
W^+= p/u_n,
\quad
\Lambda^+ = -u_n
\label{IEE_inflow}
\end{equation}
where $u_n < 0$ while at outflow with $u_n > 0$ we get the reversed situation with
\begin{equation}
W^+=
\begin{bmatrix}
u_n+p/u_n\\
u_{\tau} 
\end{bmatrix},
\quad
\Lambda^+ =
\begin{bmatrix}
u_n & 0 \\
0 & u_n 
\end{bmatrix},
\quad
W^-= p/u_n,
\quad
\Lambda^- = -u_n.
\label{IEE_outflow}
\end{equation}
By using the definitions in (\ref{IEE_inflow}) and (\ref{IEE_outflow}), it is straightforward to check for boundedness using Lemma \ref{lemma:GenBC}.
\begin{example}\label{IEE_EX} Consider the general form of boundary condition in (\ref{Gen_BC_form}). 

With Dirichlet inflow conditions on the normal and tangential velocities $u_n, u_{\tau}$ we find that 
\begin{equation}
W^- - RW^+=
\begin{bmatrix}
u_n+p/u_n\\
u_{\tau} 
\end{bmatrix}
-
\begin{bmatrix}
R_1 \\
R_2
\end{bmatrix} p/u_n=
\begin{bmatrix}
u_n \\
u_{\tau}
\end{bmatrix}
\quad \Rightarrow \quad
\begin{bmatrix}
R_1 \\
R_2
\end{bmatrix}=
\begin{bmatrix}
1 \\
0
\end{bmatrix},
\label{IEE_inflow_ex}
\end{equation}
which lead to $ \Lambda^+ - R^T  |\Lambda^- |  R=0$. Hence condition  (\ref{R_condition}) is satisfied, but not strictly, which makes the choice of $S$ in (\ref{S_condition}) irrelevant. This leads to boundedness, but not strong boundedness as defined in Proposition \ref{lemma:Matrixrelation}. A weak 
implementation require $\Sigma= |\Lambda^- |=diag( |u_n |,|u_n |)$ as specified in (\ref{Sigma_condition}).

For an outflow condition on the characteristic variable $p/u_n$, the boundary condition (\ref{Gen_BC_form}) holds with $ R=(0,0)$
and hence condition (\ref{R_condition})  holds strictly.  This leads to a strongly energy bounded solution using $S=\tilde S^{-1} \sqrt{|u_n |}$ with $ |\tilde S | \leq 1$ as  can be seen in (\ref{estimate_2}) and (\ref{final _gen_weak_result_3}) and
 required in (\ref{S_condition}). A weak  implementation require $\Sigma= |\Lambda^- |=|u_n |,$ as specified in (\ref{Sigma_condition}).
\end{example}

\subsection{The 2D shallow water equations}\label{SWEex_2D}

The 2D SWEs on  skew-symmetric form as required in Proposition \ref{lemma:Matrixrelation} and derived in  \cite{nordstrom2022linear-nonlinear} are
\begin{equation}\label{eq:swNoncons_new_skew}
    U_t +  (A U)_x + A^T U_x+ (B U)_y + B^T U_y+CU= 0,
\end{equation}
where $U=(U_1,U_2,U_3)^T=(\phi, \sqrt{\phi} u, \sqrt{\phi} v))^T$, $\phi = g h$ is the geopontential  \cite{oliger1978},  $h$ is the water height, $g$ is the gravitational constant and $(u,v)$ is the fluid velocity in $(x,y)$ direction respectively.  The Coriolis forces are included in the matrix $C$ with the function $f$ which is typically a function of latitude \cite{shallowwaterbook,whitham1974}. Note that $h>0$ and $\phi>0$ from physical considerations.  The matrices in (\ref{eq:swNoncons_new_skew}) constitute a
two-parameter family 
\begin{equation}\label{eq:swNoncons_new_matrix_ansatz_sol_A}
A =        \begin{bmatrix}
     \alpha \frac{U_2}{\sqrt{U_1}}  & (1-3 \alpha) \sqrt{U_1}                  & 0 \\
     2\alpha  \sqrt{U_1}                 & \frac{1}{2}  \frac{U_2}{\sqrt{U_1}} &  0  \\
     0                                            & 0                                                    &\frac{1}{2} \frac{U_2}{\sqrt{U_1}}
       \end{bmatrix},
       B = 
       \begin{bmatrix}
     \beta \frac{U_3}{\sqrt{U_1}}  & 0               &  (1-3 \beta) \sqrt{U_1}    \\
     0              & \frac{1}{2}  \frac{U_3}{\sqrt{U_1}} &  0  \\
     2\beta  \sqrt{U_1}                                         & 0                                                    &\frac{1}{2} \frac{U_3}{\sqrt{U_1}}
       \end{bmatrix}, 
  C = \begin{bmatrix}
     0 & 0  & 0      \\
     0 & 0  &  -f    \\
     0 & +f & 0
       \end{bmatrix}
\end{equation}
where the parameters $\alpha, \beta$ are arbitrary. (Symmetric matrices are e.g. obtained with $\alpha=\beta=1/5$.) 
\label{rem:noalpha_2d}

The  energy rate cannot depend on the free parameters $\alpha$ and $\beta$ in the matrices  $A$ and $B$ since they are not present in the original SWEs from where (\ref{eq:swNoncons_new_skew}) is derived \cite{nordstrom2022linear-nonlinear}.
By computing the boundary term, we find
\begin{equation}\label{boundarmatrix_2D}
U^T (n_1A+n_2 B)U = 
U^T
       \begin{bmatrix}
      \frac{\alpha+\beta}{2} u_n                      &  \frac{1-\alpha}{2} n_x \sqrt{U_1} &   \frac{1-\beta}{2} n_y \sqrt{U_1}     \\
     \frac{1-\alpha}{2} n_x \sqrt{U_1} & \frac{1}{2} u_n                             &    0                                                   \\
     \frac{1-\beta}{2} n_y \sqrt{U_1}   & 0                                                  &\frac{1}{2} u_n
       \end{bmatrix}
U=
U^T
       \begin{bmatrix}
      u_n                      & 0                       & 0    \\
                                  & \frac{1}{2} u_n  &    0                                                   \\
       0                         & 0                       &\frac{1}{2} u_n
       \end{bmatrix}
U
\end{equation}
and the (somewhat mysterious) dependency on the free parameters $\alpha$ and $\beta$ vanishes. The relation (\ref{boundarmatrix_2D}) seemingly indicate that we need three boundary conditions at inflow ($u_n<0$), and zero at outflow ($u_n>0$).

However, this is a nonlinear problem and as shown in  \cite{nordstrom2021linear}, it can be rewritten by changing variables and observing that $u_n=(n_1 U_2+n_2 U_3)/\sqrt{U_1}$. Reformulating  (\ref{boundarmatrix_2D}) in new variables we find
\begin{equation}\label{boundarmatrix_2D_SWE_final}
U^T (n_1A+n_2 B)U = 
U^T
       \begin{bmatrix}
      u_n                      & 0                       & 0    \\
                                  & \frac{1}{2} u_n  &    0                                                   \\
       0                         & 0                       &\frac{1}{2} u_n
       \end{bmatrix}
U=
W^T
       \begin{bmatrix}
      -\frac{1}{2 U_n \sqrt{U_1}}                     & 0                       & 0    \\
                                  & \frac{1}{2 U_n \sqrt{U_1}}   &    0                                                   \\
       0                         & 0                       &\frac{1}{2 U_n \sqrt{U_1}} 
       \end{bmatrix}
W,
\end{equation}
where $W^T=(W_1, W_2, W_3)=(U_1^2, U_1^2+U_n^2, U_n U_{\tau})$. The variables $(U_1,U_n,U_{\tau})=(\phi, \sqrt{\phi} u_n,  \sqrt{\phi} u_{\tau})$ are directed in the  normal ($U_n$) and tangential  ($U_{\tau}$) direction respectively. The relation (\ref{boundarmatrix_2D_SWE_final}) indicate
that only two boundary conditions are needed at inflow when $U_n<0$. Since we search for a minimal number of boundary conditions, we consider the formulation (\ref{boundarmatrix_2D}) for outflow, where no boundary conditions are required. To be specific, at inflow where $U_n < 0$ we find 
\begin{equation}
W^-=
\begin{bmatrix}
U_1^2+U_n^2\\
U_n U_{\tau} 
\end{bmatrix},
\quad
\Lambda^- =
\begin{bmatrix}
 \frac{1}{2 U_n \sqrt{U_1}}   & 0 \\
0 &  \frac{1}{2 U_n \sqrt{U_1}}  
\end{bmatrix},
\quad
W^+= U_1^2,
\quad
\Lambda^+ = -  \frac{1}{2 U_n \sqrt{U_1}}  
\label{SWE_inflow}.
\end{equation}
The definitions in (\ref{SWE_inflow}) can be used to check any inflow conditions for boundedness using Lemma \ref{lemma:GenBC}.

\begin{example}\label{SWE_Example} Consider the general form of boundary condition in (\ref{Gen_BC_form}).

With Dirichlet inflow conditions on $U_n, U_{\tau}$ we find
\begin{equation}
W^- - RW^+=
\begin{bmatrix}
U_1^2+U_n^2\\
U_n U_{\tau} 
\end{bmatrix}
-
\begin{bmatrix}
R_1 \\
R_2
\end{bmatrix} U_1^2 =
\begin{bmatrix}
U_n^2\ \\
U_n U_{\tau}
\end{bmatrix}
\quad \Rightarrow \quad
\begin{bmatrix}
R_1 \\
R_2
\end{bmatrix}=
\begin{bmatrix}
1\\
0
\end{bmatrix},
\label{IEE_outflow_ex}
\end{equation}
which lead to $ \Lambda^+ - R^T  |\Lambda^- |  R=0$. Hence condition  (\ref{R_condition}) is satisfied, but not strictly, which makes the choice of $S$ in (\ref{S_condition}) irrelevant (similar to the inflow case in Example \ref{IEE_EX}). This leads to boundedness, but not strong boundedness as defined in Proposition \ref{lemma:Matrixrelation}. A weak 
implementation require $\Sigma= |\Lambda^- | $ in (\ref{SWE_inflow}).

By instead specifying the characteristic variable $W^-$ directly (similar to the outflow case in Example \ref{IEE_EX}) we have $R=(0,0)^T$ and (\ref{R_condition}) holds strictly. This lead to a strongly bounded solution if
 $S=\tilde S^{-1} \sqrt{ |\Lambda^- |}$ with $ \tilde S =diag(\tilde s_1,  \tilde s_2)$ sufficiently small as required in (\ref{S_condition}).
A weak  implementation require $\Sigma= |\Lambda^- | $ in (\ref{SWE_inflow}).


\end{example} 

\subsection{The 2D compressible Euler equations}\label{CEEex_2D}

The 2D CEEs on  skew-symmetric form as required in Proposition \ref{lemma:Matrixrelation} and derived in  \cite{Nordstrom2022_Skew_Euler} are
\begin{equation}\label{eq:stab_eq}
P \Phi_t + (A  \Phi)_x + A^T  \Phi_x + (B  \Phi)_y + B^T  \Phi_y=0,
\end{equation}
where  $\Phi=(\sqrt{\rho}, \sqrt{\rho} u, \sqrt{\rho} v, \sqrt{p})^T$, $P=diag(1, (\gamma -1)/2, (\gamma -1)/2, 1)$ and
\begin{equation}\label{eq:new_matrix_finalAs}
A = \frac{1}{2}\begin{bmatrix}
      u                                    &  0  &  0 & 0 \\
      0                                    &   \frac{(\gamma -1)}{2} u  &  0 & 0                \\
      0  &   0                                                      &   \frac{(\gamma -1)}{2} u &  0             \\
      0  &  2 (\gamma -1)  \frac{\phi_4}{\phi_1} & 0  &  (2-\gamma)u           
                   \end{bmatrix}, \quad
B = \frac{1}{2}\begin{bmatrix}
      v                                    &  0  &  0 & 0 \\
      0                                    &  \frac{(\gamma -1)}{2} v  &  0 & 0                \\
      0  &   0                                                      &   \frac{(\gamma -1)}{2} v &  0             \\
      0  &  0 & 2 (\gamma -1)  \frac{\phi_4}{\phi_1}   &  (2-\gamma)v          
                   \end{bmatrix}.
 \end{equation}

By rotating the Cartesian velocities to normal and tangential velocities at the boundary, we obtain
\begin{equation}\label{boundarmatrix_contraction_rotated}
\Phi^T (n_1 \tilde A + n_2 \tilde B) \Phi = 
\Phi^T_r  \begin{bmatrix}
      \alpha^2 u_n  &  0 &  0 & 0                       \\
       0                    &  \frac{(\gamma -1)}{2}   u_n         &  0                                          &      (\gamma -1)  \frac{\phi_4}{\phi_1}                                                 \\
       0                    &   0                       &  \frac{(\gamma -1)}{2}   u_n                                                                 &  0                                                      \\
       0                    &  (\gamma -1)  \frac{\phi_4}{\phi_1}& 0 & (2-\gamma)u_n            
                   \end{bmatrix} \Phi_r,
\end{equation}
where  $\Phi=(\phi_1, \phi_2, \phi_3,  \phi_4)^T=(\sqrt{\rho}, \sqrt{\rho} u_n, \sqrt{\rho} u_{\tau}, \sqrt{p})^T$. 

The boundary term (\ref{boundarmatrix_contraction_rotated}) can be rotated to diagonal form which yield the boundary term $W^T \Lambda W$ where
\begin{equation}
W=
\begin{bmatrix}
\phi_1\\
\phi_2+2 \phi_4^2/\phi_2 \\
\phi_3 \\
\phi_4 
\end{bmatrix}
\quad 
\Lambda=
\begin{bmatrix}
      u_n  &  0 &  0 & 0                       \\
       0                    &  \frac{(\gamma -1)}{2}   u_n         &  0                                          &   0                                                \\
       0                    &   0                       &  \frac{(\gamma -1)}{2}   u_n                                                                 &  0                                                      \\
       0                    &  0& 0 & (2-\gamma)u_n     \Psi(M_n)        
                   \end{bmatrix}.
\label{CEE_def}
\end{equation}
By comparing with (\ref{boundarmatrix_contraction_rotated}) we see that the last diagonal entry is modified by the multiplication of  $\Psi(M_n)$ which is a function of the normal Mach number $M_n = u_n/c$. Explicitly we have
\begin{equation}\label{psi_def}
\Psi(M_n)= 1-\frac{2(\gamma -1)}{\gamma (2-\gamma)} \frac{1}{M_n^2}, 
\end{equation}
which switches sign at  $M_n^2=\gamma (2-\gamma)/(2(\gamma -1))$. 
\begin{remark}
As shown in  \cite{Nordstrom2022_Skew_Euler}, this yields  $|M_n|=1$ for $\gamma=\sqrt{2}$, while for  $\gamma=1.4$ we get $|M_n|=1.05$.
\end{remark}

Due to the sign shift in $\Psi$ at $M_n^2=\gamma (2-\gamma)/(2(\gamma -1))$ we get different cases for inflow where $u_n < 0$.  We find that for subsonic inflow where $u_n <0, \Psi < 0$, the relation (\ref{CEE_def}) leads to
\begin{equation}
W^-=
\begin{bmatrix}
\phi_1\\
\phi_2+2 \phi_4^2/\phi_2 \\
\phi_3
\end{bmatrix},
\quad 
\Lambda^-=
\begin{bmatrix}
      u_n  &  0 &  0                 \\
       0                    &  \frac{(\gamma -1)}{2}   u_n         &  0                                                \\
       0                    &   0                       &  \frac{(\gamma -1)}{2}   u_n                                                                                     \\
\end{bmatrix},
\quad 
W^+= \phi_4,
\quad
\Lambda^+ = (2-\gamma)u_n     \Psi(M_n).
\label{CEE_inflow}
\end{equation}
For supersonic inflow $u_n <0, \Psi >0$ we get $W^-=W$ and $\Lambda^- = \Lambda$ from relation (\ref{CEE_def}), i.e. all eigenvalues are negative. 

In the outflow case, the shift in speed can be ignored since an alternate form of (\ref{boundarmatrix_contraction_rotated}) different from (\ref{CEE_def}) exist. By contracting (\ref{boundarmatrix_contraction_rotated}) we find that 
\begin{equation}\label{contracted_CEE_BT}
\Phi^T (n_x  A + n_y B) \Phi = u_n(\phi_1^2+\frac{(\gamma -1)}{2} (\phi_2^2+\phi_3^2)+\gamma  \phi_4^2)=
\Phi^T_r  \begin{bmatrix}
      u_n  &  0 &  0 & 0                       \\
       0                    &  \frac{(\gamma -1)}{2}   u_n         &  0                                          &    0                                               \\
       0                    &   0                       &  \frac{(\gamma -1)}{2}   u_n                                                                 &  0                                                     \\
       0                    & 0 & 0 & \gamma u_n            
                   \end{bmatrix} \Phi_r,
\end{equation}
which proves that no boundary conditions are necessary in the outflow case. 
\begin{example} Consider the general form of boundary condition in (\ref{Gen_BC_form}).

With Dirichlet inflow conditions on $\phi_1,\phi_2,\phi_3$ for $\Psi(M_n) < 0$ we find using (\ref{IEE_inflow})
\begin{equation}
W^- - RW^+=
\begin{bmatrix}
\phi_1\\
\phi_2+2 \phi_4^2/\phi_2 \\
\phi_3
\end{bmatrix}
-
\begin{bmatrix}
R_1 \\
R_2  \\
R_3
\end{bmatrix} \phi_4=
\begin{bmatrix}
\phi_1\\
\phi_2 \\
\phi_3
\end{bmatrix}
\quad \Rightarrow \quad
\begin{bmatrix}
R_1 \\
R_2  \\
R_3
\end{bmatrix}=
\begin{bmatrix}
0\\
2 \phi_4/\phi_2 \\
0
\end{bmatrix},
\label{IEE_outflow_ex}
\end{equation}
which lead to $ \Lambda^+ - R^T  |\Lambda^- |  R=- |u_n|(2-\gamma + 2(\gamma-1)/(\gamma M_n^2)) < 0$. Hence condition  (\ref{R_condition}) is violated, and no bound can be found. 

By instead specifying the characteristic variable  $W^-$ directly (as for the outflow case in Example \ref{IEE_EX} and inflow case in Example \ref{SWE_Example}) we have $R=(0,0,0)^T$ and (\ref{R_condition}) holds strictly. This lead to a strongly bounded solution if
 $S=\tilde S^{-1} \sqrt{ |\Lambda^- |}$ with $ \tilde S =diag(\tilde s_1,  \tilde s_2, \tilde s_3)$ sufficiently small, see (\ref{S_condition}).
A weak  implementation require $\Sigma= |\Lambda^- | $ in (\ref{CEE_inflow}). In the outflow case, no boundary conditions are required due to (\ref{contracted_CEE_BT}).


\end{example} 

\subsection{Open questions for nonlinear boundary conditions}\label{sec:open_q}
We will end this section by discussing some open questions stemming from the nonlinear analysis above. 

\subsubsection{The number of boundary conditions in nonlinear IBVPs required for boundedness}\label{sec:open_q_number}
The boundary conditions for the SWEs and CEEs are similar in the sense that at least two {\it different} formulations of the boundary terms can be found. The minimal number of required conditions differ both in the inflow and outflow cases. One common feature is that that no outflow conditions seem to be necessary. Another similar feature is that the number of outflow conditions is independent of the speed of sound for the CEEs and the celerity in the SWE case.  Both these effects differ from what one finds in a linear analysis. 
\begin{remark}
By substituting the IEE variables $W=(u_n +p/u_n, u_{\tau},p/u_n)^T$ in (\ref{1Dprimalstab_trans_final_icomp}) with $W=(u_n, u_{\tau}, \sqrt{p})^T$ (similar to the ones used in the CEE and SWE cases) one obtains a similar situation also for the IEEs. The eigenvalues for the IEEs transform from $\Lambda=diag(u_n,u_n,-u_n)$ to $\Lambda=diag(u_n,u_n, 2u_n)$ which leads to different number of boundary conditions.
\end{remark}

\subsubsection{The effect of nonlinear boundary conditions on uniqueness and existence}\label{sec:open_q_uniqueness}
Roughly speaking, a minimal number of dissipative boundary conditions in the linear case  leads to uniqueness by the fact that it determines the normal modes of the solution \cite{Kreiss1970277,Strikwerda1977797}. The minimal number of boundary conditions can also be obtained using the energy method, see \cite{nordstrom2020}. If uniqueness and boundedness for a minimal number of boundary conditions are given, existence can be shown (e.g. using Laplace transforms or difference approximations \cite{gustafsson1995time,kreiss1989initial}). For linear IBVPs,  the number of boundary conditions is independent of the solution and only depend on known external data. For nonlinear IBVPs, that is no longer the case, and the number may change in an unpredictable way as the solution develops in time. In addition, as we have seen above, it also varies depending on the particular formulation choosen. This is confusing and raises a number of questions that we will {\it speculate} on below.


Let us consider the SWEs as an example. The two forms of the boundary terms given in (\ref{boundarmatrix_2D_SWE_final}) were

\begin{equation}\label{boundarmatrix_examples}
U^T
       \begin{bmatrix}
      u_n                      & 0                       & 0    \\
                                  & \frac{1}{2} u_n  &    0                                                   \\
       0                         & 0                       &\frac{1}{2} u_n
       \end{bmatrix}
U=
W^T
       \begin{bmatrix}
      -\frac{1}{2 U_n \sqrt{U_1}}                     & 0                       & 0    \\
                                  & \frac{1}{2 U_n \sqrt{U_1}}   &    0                                                   \\
       0                         & 0                       &\frac{1}{2 U_n \sqrt{U_1}} 
       \end{bmatrix}
W,
\end{equation}
where $W^T=(W_1, W_2, W_3)=(U_1^2, U_1^2+U_n^2, U_n U_{\tau})$ and $(U_1,U_n,U_{\tau})=(\phi, \sqrt{\phi} u_n,  \sqrt{\phi} u_{\tau})$. 
Based on the two formulations in (\ref{boundarmatrix_examples}), one may base the boundary procedure on one of the following four scenarios.
\begin{enumerate}

\item The left formulation with variable $U$ at both inflow and outflow boundaries.

\item The right formulation with variable $W$ at both inflow and outflow boundaries.

\item The left formulation with variable $U$ at inflow and the right formulation with $W$ at outflow boundaries.

\item The right formulation with variable $W$ at inflow and the left formulation with $U$ at outflow boundaries.

\end{enumerate}
Scenario 1 would in a one-dimensional setting lead to three boundary conditions all applied on the inflow boundary. Scenario 2 would also give three boundary conditions, but now two would be applied on the inflow boundary and one on the outflow boundary. Scenario 3 would lead to four boundary conditions, three on the inflow and one on the outflow boundary. Scenario 4 would only give two boundary conditions, both applied on the inflow boundary. 

If the above scenarios were interpreted in the linear sense, both Scenario 1 and 2 would determine the solution uniquely. (One of them would be a better choice than the other depending on the growth or decay of the solution away from the boundary \cite{Kreiss1970277,Strikwerda1977797}.) In scenario 3, the solution would be overspecifed, leading to loss of existence. In scenario 4, the solution would be underspecified, leading to loss of uniqueness. In summary: Scenario 1 and 2 may lead to acceptable solutions, Scenario 3 give no solution at all, while scenario 4 yield a bounded solution with limited (or no) accuracy.

However, since these results are nonlinear, the above summary is merely {\it speculative}. We do not know exactly how to interpret them, since the present nonlinear theory is incomplete. We only know that boundedness is required. It also seems likely though that scenario 1 and 2 are should be preferred over scenario 3 and 4. The speculations in this section are of course equally valid (or not valid) for the CEEs and IEEs.

\section{Nonlinear energy and entropy stability}\label{numerics}
Consider the extended version (\ref{eq:nonlin_lif}) of (\ref{eq:nonlin}) rewritten (using Einsteins summation convention) for clarity
\begin{equation}\label{eq:stab_eq}
P U_t + (A_i U)_{x_i}+A^T_i U_{x_i}+CU+L_C=0,  \quad t \geq 0,  \quad  \vec x=(x_1,x_2,..,x_k) \in \Omega.
\end{equation}
 Equation (\ref{eq:stab_eq}) is augmented with the initial condition $U(\vec x,0)=F(\vec x)$ in $\Omega$ and boundary conditions of the form (\ref{Gen_BC_form}) on $\delta\Omega$.  Furthermore $A_i=A_i(U)$, $C=C(U)$  and $P$ are  $n \times n$ matrices while $U$ and $L_C$ are $n$ vectors.  $L_C$ is the continuous lifting operator of the form (\ref{Pen_term_Gen_BC_form}) implementing the boundary conditions weakly. 
 A straightforward approximation of (\ref{eq:stab_eq}) on summation-by-parts (SBP) form in $M$ nodes is
\begin{equation}\label{EUL_Disc}
(P \otimes I_M) \vec U_t+{\bf D_{x_i}} {\bf A_i}  \vec U+{\bf A_i^T} {\bf D_{x_i}} \vec U   + {\bf C}  \vec U + {\vec L_D}=0, \quad \vec U(0) = \vec F
\end{equation}
where $\vec U=(\vec U_1^T, \vec U_2^T,...,\vec U_n^T)^T$ include approximations of  $U=(U_1,U_2,...,U_n)^T$ in each node. The discrete lifting operator ${\vec L_D} (\vec U)$ implements the boundary conditions in a similar way to  $L_C(U)$ and $\vec F$ denotes the discrete initial data with the continuous initial data injected in the nodes. The matrix elements of ${\bf A_i},{\bf C}$ are matrices with node values of the matrix elements in $A_i,C$ injected on the diagonals as exemplified below
\begin{equation}
\label{illustration}
A_i=
\begin{pmatrix}
      a_{11}   &  \ldots  & a_{1n} \\
       \vdots   & \ddots & \vdots \\
       a_{n1} &  \ldots  & a_{nn}
\end{pmatrix}, \quad
{\bf A_i} =
\begin{pmatrix}
      {\bf a_{11}}   &  \ldots  &  {\bf a_{1n}}  \\
       \vdots          & \ddots &  \vdots           \\
        {\bf a_{n1}} &  \ldots &  {\bf a_{nn}} 
\end{pmatrix}, \quad
{\bf a_{ij}} =diag(a_{ij}(x_1,y_1), \ldots, a_{ij}(x_M,y_M)).
\end{equation}
Moreover ${\bf D_{x_i}}=I_n \otimes D_{x_i}$ where $\otimes$ denotes the Kronecker product, $I_n$ is the  $n \times n$ identity matrix, $D_{x_i}=P_{\Omega}^{-1} Q_{x_i}$ are SBP difference operators, $P_{\Omega}$ is a positive definite diagonal volume quadrature matrix that defines a scalar product and norm such that  
\begin{equation}
\label{volumenorm}
(\vec U, \vec V)_{\Omega} = \vec U^TP_{\Omega}  \vec V \approx  \int\limits_{\Omega}U^T V d \Omega, \quad \text{and}  \quad  (\vec U, \vec U)_{\Omega} = \| \vec U\|^2_{\Omega}  = \vec U^TP_{\Omega}  \vec U \approx  \int\limits_{\Omega}U^T U d \Omega = \| U\|^2_{\Omega} .
\end{equation}

Following \cite{Lundquist201849} we introduce the discrete normal $\mathbf{N}=(N_1,N_2,...,N_k)$ approximating the continuous normal $\mathbf{n}=(n_1,n_2,...,n_k)$ in the $N$ boundary nodes and a restriction operator $E$ that extracts the boundary values $E\vec U$ from the total values. We also need a positive definite diagonal boundary quadrature $P_{\partial\Omega}=diag(ds_1,ds_2,...,ds_N)$ such that $\oint_{\partial \Omega}  U^T U ds \approx (E \vec U)^T P_{\partial\Omega} (\vec U)= (EU)_i^2 ds_i$. With this notation in place (again using Einsteins summation convention), the SBP constraints for a scalar variable becomes
\begin{equation}
\label{SBP_constraint}
 Q_{x_i}+Q_{x_i}^T=E^T P_{\partial\Omega} N_{i} E, 
\end{equation}
which leads to the scalar summation-by-parts formula mimicking integration-by-parts
\begin{equation}
\label{SBP_scalar_relation}
(\vec U,  D_{x_i}  \vec V) = \vec U^T  P_{\Omega} (D_{x_i}  \vec V) = - ( D_{x_i}  \vec U, \vec V)  + (E \vec U)^T P_{\partial\Omega} N_{i} (E \vec V).
\end{equation}
The scalar SBP relations in (\ref{SBP_constraint}),(\ref{SBP_scalar_relation}), correspond to the SBP formulas for a vector with $n$ variables as
\begin{equation}\label{Multi-SBP}
(\vec U, {\bf D_{x_i}} \vec V)=\vec U^T  (I_n   \otimes P_{\Omega})({\bf D_{x_i}} \vec V)= -({\bf D_{x_i}} \vec U, \vec V) + (E\vec U)^T  (I_n   \otimes P_{\partial\Omega}) N_{i}  (E\vec V).
\end{equation}

It remains to construct the discrete lifting operator $\vec L_D$ (often called the SAT term  \cite{svard2014review,fernandez2014review}) such that we can reuse the continuous analysis. We consider an operator of the form $\vec L_D =(I_n \otimes P_{\Omega})(DC) \vec L_C$.
The transformation matrix $DC$ first extracts the boundary nodes from the volume nodes, secondly permute the dependent variables from being organised as $(E\vec U_1, E\vec U_2,..,E\vec U_n)^T$  to $((E \vec U)_1, (E \vec U)_2,.., (E \vec U)_N)^T$ using the permutation matrix $P_{erm}$ 
and thirdly numerically integrate the resulting vector against the continuous lifting operator $ \vec L_C$ (now applied to the discrete solution).  More specifically we have
\begin{align}
&\vec L_D=(I_n \otimes P_{\Omega})(DC) \vec L_C,&  &\ \ \ DC=(I_n \otimes E^T) (P_{erm}) ^T (P_{\partial\Omega} \otimes I_n),& \label{generic_discrete_BT_lifting_details_1} \\ 
&\vec L_C=diag((L_C)_1,(L_C)_2,...,(L_C)_N),&             &(L_C)_j=(2 (J^- T^{-1})^T  \Sigma (\vec W^- - R \vec W^+ - S^{-1} \vec G))_j.& \label{generic_discrete_BT_lifting_details_2}
\end{align}

We can now prove the semi-discrete correspondence to Proposition \ref{lemma:Matrixrelation}.
\begin{proposition}\label{lemma:Matrixrelation_discrete}
Consider the nonlinear scheme (\ref{EUL_Disc}) with $\vec L_D $ defined in (\ref{generic_discrete_BT_lifting_details_1}) and  (\ref{generic_discrete_BT_lifting_details_2}).

 It is nonlinearly stable for $\vec G=0$ if the relations (\ref{R_condition}) and (\ref{Sigma_condition}) in Lemma \ref{lemma:GenBC} hold and the solution satisfies the estimate 
\begin{equation}\label{1Dprimalstab_disc}
\| \vec U \|_{P  \otimes P_{\Omega}}^2 \leq \| \vec F \|_{P  \otimes P_{\Omega}}^2.
\end{equation}

It is strongly nonlinearly stable for $\vec G \neq 0$ if the relations (\ref{R_condition}),(\ref{S_condition}) and (\ref{Sigma_condition}) in Lemma \ref{lemma:GenBC} hold and the solution satisfies the estimate 
\begin{equation}\label{1Dprimalstab_strong_disc}
\| \vec U \|_{P  \otimes P_{\Omega}}^2 \leq \| \vec F \|_{P  \otimes P_{\Omega}}^2 +  2 \int_0^t  \sum_{j=1,N} \lbrack \vec G^T \vec G) \rbrack_j ds_j \ dt.
\end{equation}
In (\ref{1Dprimalstab_disc}) and (\ref{1Dprimalstab_strong_disc}), $\vec F$ and $\vec G$ are external data from $F$ and $G$ injected in the nodes. 
\end{proposition}
\begin{proof}
The discrete energy method (multiply (\ref{EUL_Disc}) from the left with  $\vec U^T (I_n \otimes P_{\Omega}$) yields 
\begin{equation}\label{Disc_energy_initial}
\vec U^T (P  \otimes P_{\Omega})  \vec U_t+ (\vec U,  {\bf D_{x_i}} {\bf A_i}  \vec U)+  ({\bf A_i} \vec U, {\bf D_{x_i}} \vec U)+(\vec U,{\vec L_D)}=0,                                                       
\end{equation}
where we have used that $(I_n \otimes P_{\Omega})$ commutes with ${\bf A_i}$ (since the matrices have diagonal blocks) and that the symmetric part of  $C$ is zero. The SBP constraints (\ref{Multi-SBP}) and the notation $\vec U^T (P  \otimes P_{\Omega})  \vec U= \| \vec U \|_{P  \otimes P_{\Omega}}^2 $ simplifies (\ref{Disc_energy_initial}) to
\begin{equation}\label{Disc_energy_final}
\dfrac{1}{2} \dfrac{d}{dt} \| \vec U \|_{P  \otimes P_{\Omega}}^2 
+ \vec U^T (I_n  \otimes E^T P_{\partial\Omega} N_{i} E )  {\bf A_i} \vec U  + (\vec U,{\vec L_D}) =0.
\end{equation}

The semi-discrete energy rate in (\ref{Disc_energy_final}) mimics the continuous energy rate in the sense that only boundary terms remain. To make use of the already performed continuous energy analysis, we expand the boundary terms and exploit the diagonal form of  $P_{\partial\Omega}$. The result is
\begin{equation}\label{Disc_energy_final_BT}
\vec U^T (I_n  \otimes E^T P_{\partial\Omega} N_{x_i} E )  {\bf A_i} \vec U = \sum_{j=1,N} \lbrack  (E  \vec U)^T (N_{i}  {\bf A_i})  (E  \vec U) \rbrack_j ds_j.
\end{equation}
The relation (\ref{Disc_energy_final_BT}) mimics the continuous result (\ref{1Dprimalstab_trans}) in each of the  $N$ boundary nodes. Next, the continuous transformation formula applied to the discrete solution yields $W_i = (T^{-1}(E \vec U))_i $ and hence
\begin{equation}\label{discrete_BT_diagonal}
\sum_{j=1,N}  \lbrack  (E  \vec U)^T (N_{i}  {\bf A_i})  (E  \vec U) \rbrack_j ds_j = \sum_{j=1,N}  \lbrack \vec W^T   \Lambda   \vec W \rbrack_jds_j.
\end{equation}
The  discrete boundary terms (\ref{discrete_BT_diagonal}) now have the same form as the continuous ones in (\ref{1Dprimalstab_trans_final}). By using  (\ref{generic_discrete_BT_lifting_details_1}) and (\ref{generic_discrete_BT_lifting_details_2}) we find that
\begin{equation}\label{discrete_BT_lifting}
(\vec U,{\vec L_D})= \sum_{j=1,N}  \lbrack 2  (\vec W^-)^T  \Sigma(\vec W^- - R \vec W^+ - S^{-1} \vec G) \rbrack_j ds_j.
\end{equation}
The combination of  (\ref{Disc_energy_final})-(\ref{discrete_BT_lifting}) leads to the final form of the energy rate
\begin{equation}\label{Disc_energy_final_final}
\dfrac{1}{2} \dfrac{d}{dt} \| \vec U \|_{P  \otimes P_{\Omega}}^2 + \sum_{j=1,N}  \lbrack  \vec W^T   \Lambda \vec W + 2  (\vec W^-)^T  \Sigma(\vec W^- - R \vec W^+ - S^{-1} \vec G) \rbrack_j ds_j=0.
\end{equation}
By using (\ref{R_condition})-(\ref{Sigma_condition}) in Lemma \ref{lemma:GenBC}, the estimates (\ref{1Dprimalstab_disc}) and (\ref{1Dprimalstab_strong_disc}) follow by using the same technique that was used for the continuous estimates in the proof of Lemma \ref{lemma:GenBC}.
\end{proof}

\section{Summary}\label{sec:conclusion}
In this paper we have completed the general stability theory for nonlinear skew-symmetric hyperbolic problems partly developed in \cite{nordstrom2022linear-nonlinear,Nordstrom2022_Skew_Euler}, by adding the analysis of nonlinear boundary conditions. In  \cite{nordstrom2022linear-nonlinear, Nordstrom2022_Skew_Euler} we focused on the skew-symmetric property {\it assuming}  that boundary conditions leading to an energy bound were available. In this article we derive these boundary conditions explicitly, and show how to implement them in a provable stable way using summation-by-parts formulations and weak boundary procedures. We exemplify the general procedure for the most important equations in computational fluid dynamics: the shallow water equations, the incompressible Euler equations and the compressible Euler equations.



\section*{Acknowledgments}

Jan Nordstr\"{o}m was supported by Vetenskapsr{\aa}det, Sweden [award no.~2018-05084 VR and 2021-05484 VR] and the Swedish e-Science Research Center (SeRC).

\bibliographystyle{elsarticle-num}
\bibliography{References_Jan,References_andrew,References_Fredrik}

\end{document}